\documentclass[a4paper,11pt]{article}
\usepackage[a4paper,left=10mm,right=10mm]{geometry}

\usepackage{xspace}
\usepackage{tikz}
\usepackage{caption}
\usepackage{subcaption}
\usetikzlibrary{mindmap}
\usepackage{amsmath}
\usepackage{bbold}
\usepackage{enumitem}
\usepackage[9pt]{extsizes}
\usepackage{url}
\usepackage{lipsum}
\usepackage{graphicx}
\usepackage{amsthm}
\usepackage{amssymb}

\newtheorem{theorem}{Theorem}
\newtheorem{definition}{Definition}

\newtheorem{lemma}{Lemma}
\newtheorem{exmp}{Example}

\newcommand{\ppp}{P^{++}}
\newcommand{\pp}{P^{+}}
\newcommand{\hr}{$H$-representation}
\newcommand{\hp}{$\mathcal{H}(P)$}
\newcommand{\vp}{$\mathcal{V}(P)$}
\newcommand{\hq}{$\mathcal{H}(Q)$}

\newcommand{\vr}{$V$-representation}
\newcommand{\lrs}{\emph{lrslib}\xspace}

\newcommand{\cdd}{\emph{cddlib}\xspace}

\usepackage{amsmath}
\usepackage{amssymb}
\usepackage{latexsym}
\usepackage{graphicx}
\usepackage{color}
\usepackage{hyperref}
\usepackage{slashbox}
\usepackage{algorithm}
\usepackage{algpseudocode}

\definecolor{darkblue}{rgb}{0,0,0.6}

\include{./mathlib}
\include{./macros}

\title{$HV$-symmetric polyhedra and bipolarity}
\author{David Avis}

\begin{document}
\maketitle

\begin{abstract}
A polyhedron is pointed if it contains at least one vertex.
Every pointed polyhedron $P$ in $R^n$ can be described by an \hr \\
\hp~consisting of half spaces or equivalently by a \vr~\vp~ \\
consisting of the
convex hull of a set of vertices and extreme rays. 
We can define 
matrices $H(P)$ and $V(P)$, each with $n+1$ columns,
that encode these representations. 
Define polyhedron $Q$ by setting $H(Q)=V(P)$.
We show that $Q$ is the polar of $P$.
Call $P$ $HV$-symmetric if $Q$ is pointed and $V(Q)$ in turn encodes \hp.
It is well known and often stated that
polytopes that
contain the origin in their interior 
and pointed polyhedral cones are $HV$-symmetric.
We show here that,
more generally, a 
pointed polyhedron with pointed polar is $HV$-symmetric
if and only if
it contains the origin. We prove this 
using Minkowski's bipolar equation and discuss implications
for the vertex and facet enumeration problems.

\noindent{}\textbf{Keywords:} Polyhedra, bipolarity, $HV$-symmetric, vertex enumeration, facet enumeration.
\end{abstract}

We give basic definitions here,
for a comprehensive treatment of the background material 
the reader is referred to Barvinok \cite{Bar}, Schrijver \cite{Schrijver} or Ziegler
\cite{Ziegler}. 
Let $P$ be a non-empty possibly unbounded rational polyhedron in $R^n$ 
defined by its
{\em\hr}
\begin{equation}
\label{hrep}
\mathcal{H}(P):~~~~~~b + Ax \ge 0 ~~~~~~
Bx \ge 0 
\end{equation} 
for an $m_A \times n$ matrix $A$, $m_B \times n$ matrix $B$ and
a column vector $b$ with non-zero components, all with rational entries. 
If $P$ is full dimensional and (\ref{hrep}) contains
no redundant inequalities this representation
is unique up to permuting rows and/or scaling any row by a positive constant.
In this case each inequality supports a {\em facet} of $P$.
If $P$ is not full dimensional the representation (\ref{hrep}) contains
equations and is not unique. A good example is the $n$-city Travelling
Salesman Polytope which contains $n$ linearly independent equations
and has more than one standard~\hr. 

We will encode \hp~
by the
$(m_A+m_B) \times (n+1)$ matrix
\begin{equation}
\label{hr}
H(P)~=~
\begin{bmatrix}
b& A \\
0_{m_B} & B
\end{bmatrix}
\end{equation}
where $0_{m_B}$ is a column vector of $m_B$ zeroes.
If $0 \in P$ then $b > 0$ and by scaling we can replace $b$
by $1_{m_A}$, a column of $m_A$ ones.

A non-empty polyhedron $Q$ that contains no line is called {\em pointed}.
If $Q$ contains a line
it can be uniquely represented
$Q=L+P$ where $L$ is a linear space and $P$ is either 
empty or a pointed polyhedron
in the linear space orthogonal to $L$ (e.g. Sec. 8.2 in \cite{Schrijver}).
A point $x \in P$ is a { \em vertex} of $P$ if it is the unique solution
of $n$ inequalities in (\ref{hrep}) expressed as equations. 
It is known that a non-empty polyhedron is pointed if and only if
it contains at least one vertex (Sec. 8.5 in \cite{Schrijver}).
If $P$ is unbounded it contains {\em extreme rays}.
An extreme ray emanates from a vertex $x \in P$
and is generated by a non-zero $r \in R^n$
for which $x + tr \in P$ for any scalar $t \ge 0$. 
If $r$ generates an extreme ray we must have $Ar \ge 0$ and $Br \ge 0$,
with $n-1$ inequalities satisfied as linearly independent equations. 
This set of vertices and extreme rays give a {\em \vr} \vp. 
In this note we consider only polyhedra with non-empty \vp,
in other words pointed polyhedra.

For a pointed polyhedron $P$ we define
an $m_S \times n$ matrix $S$ containing its {\em vertices} and 
an $m_R \times n$ matrix $R$ containing its {\em extreme rays}.
The Minkowski-Weyl theorem 
states that $P$ can be equivalently represented by
$H(P)$ and by its \vr, where:
\[
P = conv(S) + cone(R).
\]
The general form of the theorem applies also to
non-pointed $P$ in which case, as we saw above,
a linear space $L$ is also required
in the \vr. We do not consider this case in this note.

If $m_R=0$ $P$ is a polytope and if $m_S=1$ it is a cone.
We encode \vp~ by the $(m_S+m_R) \times (n+1)$ matrix
\begin{equation}
\label{vr}
V(P)~=~
\begin{bmatrix}
1_ {m_S}& S \\
0_{m_R} & R
\end{bmatrix}
\end{equation}
An interpretation of $V(P)$ is that for any $x \in P$ there exist
non-negative $\lambda \in R^{m_S}$ and $\mu \in R^{m_R}$ such
that $[1,x]= [\lambda, \mu ] V(P)$.
The representation (\ref{vr}) is non-empty and unique, up to permuting
rows and/or scaling rows of $R$ by a positive constant, for each pointed
polyhedron $P$ (Sec. 8.6 in \cite{Schrijver}).
If $P$ is not pointed then, for the purposes of this note,
$V(P)$ is undefined.

Each facet of $P$ is defined by a set of $n$ vertices and/or rays
which are affinely independent and span it.
Computing \vp~ from \hp~ is called the {\em vertex enumeration
problem} and the reverse transformation the {\em facet enumeration problem}.
The encodings in (\ref{hr}) and (\ref{vr}) are typical of those used in 
open source software for these problems
such as \cdd\footnote{\url{https://github.com/cddlib/cddlib}}
and \lrs\footnote{\url{https://cgm.cs.mcgill.ca/~avis/C/lrs.html}}.

Any matrix of the form (\ref{vr}) can be used to define two
(usually) different polyhedra depending on whether it is interpreted as
an \hr~ or a \vr. 
Accordingly define a polyhedron $Q$ by setting 
$H(Q)=V(P)$ so that we have              
\begin{equation}
\label{hrq}
\mathcal{H}(Q):~~~~~1_{m_S} + Sx \ge 0 ~~~~~ Rx \ge 0
\end{equation}
Note that $0 \in Q$.

\begin{definition}
\label{def}
Let $P$ be a pointed polyhedron whose \vp~is encoded by $V(P)$ and 
let $Q$ be the polyhedron with \hq~encoded by $V(P)$.
If $V(Q)$~encodes an \hp~of $P$
we call $P$ $HV$-symmetric.
\end{definition}
Note that if $Q$ is not pointed then $V(Q)$ is undefined and so $P$
is not $HV$-symmetric.
On the other hand if $P$ is $HV$-symmetric then $Q$ is pointed and
so $Q$ is $HV$-symmetric also.
The following two examples illustrate the definition.
\begin{exmp}
\label{eg1}
Consider the wedge $P$ in $R^2$ with 
vertex $(1,1)$ and extreme rays $(1,0), (0,1)$
which has an 
\hr
\begin{align*}
-1 + x_1~~~~~~ &\ge 0 \\
-1 ~~~~~~+ x_2  &\ge 0 .
\end{align*}
We have
\begin{equation}
\nonumber
H(P) =
\begin{bmatrix}
-1 &  1& 0\\
-1 &  0&  1 
\end{bmatrix}
~~~~H(Q):= V(P)=
\begin{bmatrix}
1 &  1&  1  \\
0 & 1&  0 \\
0 &  0&  1
\end{bmatrix}
~~~~ V(Q)=
\begin{bmatrix}
1 &  0&  0  \\
0 & 1&  0 \\
0 &  0&  1
\end{bmatrix}
\end{equation}
Since $V(Q)$ is not an encoding of an \hp~of $P$, $P$ is not $HV$-symmetric.
Note that the origin $(0,0)$ is not contained in $P$.
$\square$
\end{exmp}

\begin{exmp}
\label{eg2}
Consider the pyramid $P$ in $R^3$ with base containing
vertices $(0,1,0), (-1,0,0), (1,0,0)$ and apex $(0,0,1)$ which has
\hr
\begin{align*}
1 + x_1 -x_2 -x_3 &\ge 0 \\
1 - x_1 -x_2 -x_3 &\ge 0 \\
~~~~~~~~~~~~~~~x_3 &\ge 0 \\
~~~~~~~~~    ~x_2~~~~~ &\ge 0.
\end{align*}
$P$ is $HV$-symmetric since we have
\begin{equation}
\nonumber
H(Q) := V(P) =
\begin{bmatrix}
1 &  0&  1 & 0 \\
1 & -1&  0 & 0 \\
1 &  1&  0 & 0 \\
1 &  0&  0 & 1 
\end{bmatrix}
~~~~V(Q) =
\begin{bmatrix}
1 &  1& -1 & -1\\
1 & -1& -1 & -1 \\
0 &  0&  0 & 1 \\
0 &  0&  1 & 0 
\end{bmatrix}
=H(P)
\end{equation}
Note that the origin $(0,0,0)$ is on the boundary of $P$, 
the midpoint of vertices
$(-1,0,0)$ and $(1,0,0)$, and that $Q$ is unbounded 
since it has two extreme rays.
On the other hand the origin is interior to $Q$ and $P$ is bounded.
$\square$
\end{exmp}
We observe from Example \ref{eg1} that
if the origin is outside of $P$ then there must be some facet of the form
$-1+ax \ge 0$. Matrix $H(P)$ has a negative number in column one and cannot
be a \vr~
so $P$ is not $HV$-symmetric. 
We will prove that in all other cases (e.g. Example \ref{eg2}) $P$ is
$HV$-symmetric.

We will make use of a basic result on the polarity of convex sets.
Given any polyhedron $P$ in $R^n$ its polar\footnote{ 
The polar is often written as $P^o= \{ z \in R^n : zx \le 1~~ \forall x \in P \}$, so that $\pp = - P^o$. } 
\begin{equation}
\label{polar}
P^+ = \{ z \in R^n : 1 + zx \ge 0~~ \forall x \in P \}.
\end{equation}
According to Fenchel's historical survey \cite{Fen} (Sec. 10), 
polarity dates back at least to the work of Minkowski.
The connection of polarity 
to the $HV$ conversion problem is given by the
following result which shows that $Q$ as defined above is $\pp$.
For the case where $P$ contains the origin this is Thm. 9.1(iv)
of \cite{Schrijver}. This condition is not necessary and
we give a proof 
of this more general result for completeness.
\begin{lemma}
\label{fun}
Let $P$ be a pointed polyhedron with \vp~encoded as $V(P)$ (\ref{vr}). Then
\begin{equation}
\label{lem}
\pp = Q := \{ z \in R^n : 1 + vz \ge 0, \forall v \in S,  rz \ge 0, \forall r \in R \}
\end{equation}
and so $V(P)$ is an encoding of $\mathcal{H}(P^+)$.
\end{lemma}
\begin{proof}
\noindent
($\pp \subseteq Q $)
Choose any $z \in \pp$.
For any $v \in S, r \in R, t \ge 0$, we have $v+tr \in P$.
Therefore by (\ref{polar}) $1+vz+trz \ge 0$ for all $t \ge 0$ and hence
$1 + vz \ge 0$ and $rz \ge 0$. Therefore $z \in Q$.

\noindent
($Q \subseteq \pp $)
Choose any $z \in Q$ and $x \in P$. We can write
\begin{align}
\nonumber
x &= \sum_{v \in S} \lambda_v v + \sum_{r \in R} \mu_r r,~~~~ 1_{|S|} \lambda =1, ~\lambda \ge 0,~ \mu \ge 0 \\
\nonumber
1 + zx &=1 + z(\sum_{v \in S} \lambda_v v + \sum_{r \in R} \mu_r r) \\
\nonumber
       &= \sum_{v \in S} \lambda_v( 1 + vz) + \sum_{r \in R} \mu_r r z \\
\nonumber
       & \ge 0
\end{align}
since $1 + vz \ge 0, \forall v \in S,~  rz \ge 0, \forall r \in R$.
$\square$
\end{proof}
A consequence of the lemma is that $P$ is bounded
if and only if the origin is interior to $\pp$.
Even though the output $V(P)$ of an $H$ to $V$ conversion always encodes 
$\mathcal{H}(P^+)$,
it is not true in general that the output $H(P)$
of the reverse conversion 
is an encoding of $\mathcal{V}(P^+)$ (see Example \ref{eg1}).
To see when this happens we need a slight extension of a  
fundamental result known as the {\em bipolar equation} which
states that if the origin is interior to $P$ then $\ppp=P$.
Our proof 
follows that of Barvinok\cite{Bar} (Chapter IV, Thm. 1.2)
extended to remove the condition that the origin is interior.
We denote the {\em closure} of a
set $X \subseteq R^n$ by $cl(X)$.

\begin{lemma}[Extended bipolar equation\footnote{The lemma and proof
in fact apply to any closed convex set. An even more
general form requiring neither convexity nor closure is given
in Hoheisel Thm. 3.73 \cite{Hoheisel}}]
\label{bipolar}
If $P$ is a pointed polyhedron then 
\begin{equation}
\ppp=cl(conv(P \cup \{0\})).
\end{equation}
\end{lemma}
\begin{proof}
Define $P_0=conv(P \cup \{0\})$. By the definition of polarity $0 \in \ppp$.
Next for any $x \in P$ we have $1+zx \ge 0$ for all $z \in \pp$ and so
$x \in \ppp$. It follows by convexity that $P_0 \subseteq \ppp$.
Since $\ppp$ is closed it follows that $cl(P_0) \subseteq \ppp $.

On the other hand consider
any given $u \notin cl(P_0)$.
By the separating hyperplane theorem (see e.g. \cite{Bar} Thm. III.1.3)
there is a vector $c \neq 0$ and scalar $d$ such that $d +cx > 0$
for all  $x \in P_0$
and $d+cu < 0$. Since $0 \in P_0$ we have $d>0$.
Set $b = d^{-1}c$. We have $1 + bx > 0$ for all $x \in P_0$ and so 
$1 + bx > 0$ for all $x \in P$ and hence $b \in \pp$.
However $1 + bu < 0$ which implies that $u \notin \ppp$.
\end{proof}
We may now state and prove our main result on $HV$-symmetry.
\begin{theorem}
\label{int}
The following statements are equivalent for pointed polyhedra $P$
and $\pp$:
\begin{enumerate}[nosep,label=(\alph*)]
\item
$0 \in P$
\item
$P = \ppp$
\item
$V(P^+)$ encodes $\mathcal{H}(P)$
\item
$P$ is $HV$-symmetric.   
\end{enumerate}
\end{theorem}
\begin{proof}
(a) $\Rightarrow$ (b): from Lemma \ref{bipolar} since $P=cl(conv(P \cup \{0\}))$.

(b) $\Rightarrow$ (c): Since $\pp$ is pointed by assumption,
Lemma \ref{fun} 
states that $V(\pp)$ encodes $\mathcal{H}(\ppp)=$\hp.

(c) $\Rightarrow$ (d): Lemma \ref{fun} states that $V(P)$
encodes $H(\pp)$, so in Definition \ref{def} we have $Q=P^+$.
Substituting in (c) we have that $V(Q)$ encodes $\mathcal{H}(P)$ 
and so $P$ is $HV$-symmetric.

(d) $\Rightarrow$ (a): By contrapositive, if $0 \notin P$ then $P$ has a facet
of the form $-1 + ax \ge 0$ and cannot be $HV$-symmetric.
\end{proof}

The usual method for computing $H(P)$ from $V(P)$ is by lifting
$P$ to a cone $C$ by prepending a column of zeroes to $V(P)$. 
Because cones are $HV$-symmetric the resulting matrix can
be used in an \hr~to \vr~conversion for $C$.
The output extreme rays are interpreted as facets of $P$.
By the above theorem, if $0 \in P$ the lifting is not necessary
and an $H$ to $V$ conversion can be applied directly to the matrix
$V(P)$. This has computational ramifications for pivot based algorithms,
such as reverse search which is used in \lrs, 
since the number of pivots and hence bases generated can be completely
different for the lifted and unlifted inputs.
Empirical evidence in \cite{AV98} showed that, surprisingly, the lifted problems
often had considerably fewer bases than the unlifted ones and hence
significantly faster running times. 
Determining conditions for when lifting reduces the number of bases
is an interesting open problem. 

Finally we discuss the asymmetry between \hr~and \vr~and the role
of the origin in being $HV$-symmetric. This is due to the fact that the
`hyperplane at infinity', $1 + 0_n x \ge 0$, coded as 
$[ 1~~0_n ]$ in (\ref{hr}), is valid for any \hr. To obtain
isomorphism with a \vr~it is necessary that the origin, also coded
as $[ 1~~ 0_n ]$, be included (perhaps redundantly) in $V(P)$.
That is, $0 \in P$.

\section*{Acknowledgment}
The author is grateful to Komei Fukuda for many useful discussions 
and for encouraging him to
write this note.
Detailed comments by two anonymous referees lead to substantial 
improvements in the
presentation.
This research was supported by JSPS
Kakenhi Grants
 20H00579,  
 20H00595,  
 20H05965,  
 22H05001  
 and 23K11043.  


\end{document}